\documentclass[12 pt]{amsart}
\usepackage{amssymb,latexsym,amsmath,amscd,amsthm,graphicx, color}
\usepackage[all]{xy}
\usepackage{pgf,tikz}
\usepackage{mathrsfs}
\usepackage{cite}
\usetikzlibrary{arrows}
\definecolor{uuuuuu}{rgb}{0.26666666666666666,0.26666666666666666,0.26666666666666666}
\definecolor{xdxdff}{rgb}{0.49019607843137253,0.49019607843137253,1.}
\definecolor{ffqqqq}{rgb}{1.,0.,0.}

\definecolor{uuuuuu}{rgb}{0.26666666666666666,0.26666666666666666,0.26666666666666666}
\definecolor{qqwuqq}{rgb}{0.,0.39215686274509803,0.}
\definecolor{zzttqq}{rgb}{0.6,0.2,0.}
\definecolor{xdxdff}{rgb}{0.49019607843137253,0.49019607843137253,1.}
\definecolor{qqqqff}{rgb}{0.,0.,1.}
\definecolor{cqcqcq}{rgb}{0.7529411764705882,0.7529411764705882,0.7529411764705882}

\definecolor{uuuuuu}{rgb}{0.26666666666666666,0.26666666666666666,0.26666666666666666}
\definecolor{qqwuqq}{rgb}{0.,0.39215686274509803,0.}
\definecolor{zzttqq}{rgb}{0.6,0.2,0.}
\definecolor{xdxdff}{rgb}{0.49019607843137253,0.49019607843137253,1.}
\definecolor{qqqqff}{rgb}{0.,0.,1.}
\definecolor{cqcqcq}{rgb}{0.7529411764705882,0.7529411764705882,0.7529411764705882}

\theoremstyle{plain}

\newtheorem{theorem}[subsection]{Theorem}

\newtheorem{lemma}[subsection]{Lemma}

\newtheorem{prop}[subsection]{Proposition}
\newtheorem{defi}[subsection]{Definition}
\theoremstyle{definition}

\newtheorem{cor}[subsection]{Corollary}

\newtheorem{remark}[subsection]{Remark}

\newtheorem{example}[subsection]{Example}

\newcommand{\uu}{\cup}
\newcommand{\ii}{\cap}
\newcommand{\UU}{\bigcup}
\newcommand{\II}{\bigcap}

\newcommand{\ci}{\subseteq}
\newcommand{\sci}{\subset}
\newcommand{\es}{\emptyset}
\newcommand{\set}[1]{\{#1\}}

\newcommand{\ga}{\alpha}
\newcommand{\gb}{\beta}

\newcommand{\gd}{\delta}
\renewcommand{\gg}{\gamma}

\newcommand{\go}{\omega}

\newcommand{\gs}{\sigma}
\newcommand{\gt}{\tau}

\newcommand{\tit}{\textit}

\newcommand{\D}[1]{\mathbb{#1}}

\newcommand{\te}{\text}

\begin{document}

To appear, Israel Journal of Mathematics 
\title{Optimal quantization for the Cantor distribution generated by infinite similutudes}

\author{Mrinal Kanti Roychowdhury}
\address{School of Mathematical and Statistical Sciences\\
University of Texas Rio Grande Valley\\
1201 West University Drive\\
Edinburg, TX 78539-2999, USA.}
\email{mrinal.roychowdhury@utrgv.edu}

\thanks{The research of the author was supported by U.S. National Security Agency (NSA) Grant H98230-14-1-0320}

\date{}
\maketitle

\pagestyle{myheadings}\markboth{Mrinal Kanti Roychowdhury}{Optimal quantization for the Cantor distribution generated by infinite similutudes}

\begin{abstract}  Let $P$ be a Borel probability measure on $\mathbb R$ generated by an infinite system of similarity mappings $\{S_j : j\in \mathbb N\}$ such that $P=\sum_{j=1}^\infty \frac 1{2^j} P\circ S_j^{-1}$, where for each $j\in \mathbb N$ and $x\in \mathbb R$, $S_j(x)=\frac 1{3^{j}}x+1-\frac 1 {3^{j-1}}$. Then, the support of $P$ is the dyadic Cantor set $C$ generated by the similarity mappings $f_1, f_2 : \mathbb R \to \mathbb R$ such that  $f_1(x)=\frac 13 x$ and $f_2(x)=\frac 13 x+\frac 23$ for all $x\in \mathbb R$. In this paper, using the infinite system of similarity mappings $\{S_j : j\in \mathbb N\}$ associated with the probability vector $(\frac 12, \frac 1{2^2}, \cdots)$, for all $n\in \mathbb N$, we determine the optimal sets of $n$-means and the $n$th quantization errors for the infinite self-similar measure $P$. The technique obtained in this paper can be utilized to determine the optimal sets of $n$-means and the $n$th quantization errors for more general infinite self-similar measures.
\end{abstract}

\section{Introduction}
 The history of the theory and practice of quantization dates back to 1948, although similar ideas had appeared in the literature in 1897 (see \cite{S}). It is used in many applications such as signal processing and telecommunications, data compression,  pattern recognitions and cluster analysis (for details see \cite{GG, GN}). It is also closely connected with centroidal Voronoi tessellations. Let $\D R^d$ denote the $d$-dimensional Euclidean space, $\|\cdot\|$ denote the Euclidean norm on $\D R^d$ for any $d\geq 1$, and $P$ be a Borel probability measure on $\D R^d$. For a finite set $\ga \sci \D R^d$ and $a\in \ga$, by $M(a|\ga)$ we denote the set of all elements in $\D R^d$ which are nearest to $a$ among all the elements in $\ga$, i.e.,
\[M(a|\ga)=\set{x \in \D R^d : \|x-a\|=\min_{b \in \ga}\|x-b\|}.\]
$M(a|\ga)$ is called the \tit{Voronoi region} generated by $a\in \ga$. On the other hand, the set $\set{M(a|\ga) : a \in \ga}$ is called the \tit{Voronoi diagram} or \tit{Voronoi tessellation} of $\D R^d$ with respect to the set $\ga$.
\begin{defi} \label{defi00}
A set $\ga\sci \D R^d$ is called a \tit{centroidal Voronoi tessellation} (CVT) with respect to a probability distribution $P$ on $\D R^d$,  if it satisfies the following two conditions:

$(i)$ $P(M(a|\ga)\ii M(b|\ga))=0$ for $a, b\in \ga$, and $a \neq b$;

$(ii)$ $E(X : X \in M(a|\ga))=a$ for all $a\in \ga$,

where $X$ is a random variable with distribution $P$, and $E(X : X \in M(a|\ga))$ represents the conditional expectation of the random variable $X$ given that $X$ takes values in $M(a|\ga)$.
\end{defi}
For details about CVT and its application one can see \cite{DFG}. If $\ga$ is a finite set, the error $\int \min_{a \in \ga} \|x-a\|^2 dP(x)$ is often referred to as the \tit{cost,} or \tit{distortion error} for $\ga$ with respect to the probability measure $P$, and is denoted by $V(\ga):= V(P; \ga)$. On the other hand, $\inf\set{V(P; \ga) :\alpha \subset \mathbb R^d, \text{ card}(\alpha) \leq n}$ is called the \tit{$n$th quantization error} for the probability measure $P$, and is denoted by $V_n:=V_n(P)$. If $\int \| x\|^2 dP(x)<\infty$, then there is some set $\ga$ for which the infimum is achieved (see \cite{ GKL, GL1, GL2}). Such a set $\ga$ for which the infimum occurs and contains no more than $n$ points is called an \tit{optimal set of $n$-means}. Elements of an optimal set of $n$-means are called \tit{optimal quantizers}.  In some literature it is also refereed to as \tit{principal points} (see \cite{MKT}, and the references therein). To see some work on optimal sets of $n$-means one is refereed to \cite{DR, GL3, R1, R2, RR}. It is known that for a continuous probability measure an optimal set of $n$-means always has exactly $n$-elements (see \cite{GL2}). For a Borel probability measure $P$ on $\D R^d$, an optimal set of $n$-means forms a CVT with $n$-means ($n$-generators) of $\D R^d$; however, the converse is not true in general (see \cite{DFG, R}). A CVT with $n$-means is called an \tit{optimal CVT with $n$-means} if the generators of the CVT form an optimal set of $n$-means with respect to the probability distribution $P$. Let us now state the following proposition (see \cite[Chapter~6 and Chapter~11]{GG} and \cite[Section~4.1]{GL2}).
\begin{prop} \label{prop10}
Let $\ga$ be an optimal set of $n$-means for a continuous Borel probability measure $P$ on $\D R^d$. Let $a \in \ga$ and $M (a|\ga)$ be the Voronoi region generated by $a\in \ga$.
Then, for every $a \in\ga$,
$(i)$ $P(M(a|\ga))>0$, $(ii)$ $ P(\partial M(a|\ga))=0$, $(iii)$ $a=E(X : X \in M(a|\ga))$, and $(iv)$ $P$-almost surely the set $\set{M(a|\ga) : a \in \ga}$ forms a Voronoi partition of $\D R^d$.
\end{prop}
Let $C$ be the Cantor set generated by the contractive similarity mappings $f_1$ and $f_2$ given by $f_1(x)=\frac 13x $ and $f_2(x)=\frac 13 x+\frac 23$ for all $x\in \D R$. Then, $C$ satisfies \[C=\UU_{\go \in \set{1, 2}^{\infty}} \II_{n=1}^\infty f_{\go|_n}([0, 1]),\]
where for $\go:=\go_1\go_2\cdots \in \set{1, 2}^\infty$, $f_{\go|_n}=f_{\go_1}\circ f_{\go_2}\circ \cdots \circ f_{\go_n}$. Notice that for any $\go:=\go_1\go_2\cdots \in \set{1, 2}^\infty$ and $n\in \D N$, $\go|_n:=\go_1\go_2\cdots \go_n\in \set{1, 2}^\ast$, where $\set{1, 2}^\ast$ denotes the set of all words over the alphabet $\set{1, 2}$ including the empty word $\es$.
Define $\mu:=\frac 12 \mu \circ f_1^{-1}+\frac 12 \mu\circ f_2^{-1}$. Then, $\mu$ is a self-similar measure on $\D R$ such that $\mu(C)=1$.
\begin{defi}  \label{def1} For $n\in \D N$ with $n\geq 2$ let $\ell(n)$ be the unique natural number with $2^{\ell(n)} \leq n<2^{\ell(n)+1}$. For $I\sci \set{1, 2}^{\ell(n)}$ with card$(I)=n-2^{\ell(n)}$ let $\gb_n(I)$ be the set consisting of all midpoints $a_\gs$ of intervals $f_\gs([0, 1])$ with $\gs \in \set{1,2}^{\ell(n)} \setminus I$ and all midpoints $a_{\gs 1}$, $a_{\gs 2}$ of the basic intervals of $f_\gs([0, 1])$ with $\gs \in I$. Formally, $\gb_n(I)=\set{a_\gs : \gs \in \set{1,2}^{\ell(n)} \setminus I} \uu \set{a_{\gs 1} : \gs \in I} \uu \set {a_{\gs 2} : \gs \in I}.$
\end{defi}
In \cite{GL3}, Graf and Luschgy showed that $\gb_n(I)$ forms an optimal set of $n$-means for the probability distribution $\mu$, and the $n$th quantization error is given by
\[V (\gb_n(I))=\frac{1}{18^{\ell(n)}}\cdot \frac 18 \Big(2^{\ell(n)+1}-n+\frac 19(n-2^{\ell(n)})\Big).\]
Let $\set{S_j}_{j=1}^\infty$ be an infinite collection of contractive similitudes on $\D R$ such that $S_j(x)=\frac 1 {3^j} x+1-\frac{1}{3^{j-1}}$ for all $x\in\D R$ and all $j\in \D N$. $S_j$ has the similarity ratio $s_j$, where $s_j=\frac 1{3^j}$, for each $j\in \D N$. Let $\pi$ be the coding map from the symbol or coding space $\D N^\infty$ into $[0, 1]$ such that
\[\set{\pi(\go)}=\II_{n=1}^\infty S_{\go|_n}([0, 1]),\]
where for $\go:=\go_1\go_2\cdots \in \D N^\infty$, $\go|_n:=\go_1\go_2\cdots\go_n$, and $S_{\go|_n}=S_{\go_1}\circ S_{\go_2}\circ \cdots \circ S_{\go_n}$.
Then, $\pi$ is a continuous map from the coding space onto the set $J$ given by
\[J:=\pi(\D N^\infty)=\UU_{\go \in \D N^\infty} \II_{n=1}^\infty S_{\go|_n}([0, 1]),\]
which is called the limit set of the infinite iterated function system (IFS) $\set{S_j}_{j=1}^\infty$.
Observe that $J$ satisfies the natural invariance equality: $J=\uu_{j=1}^\infty S_j(J)$. Due to an infinite iterated function system, the limit set $J$ is not necessarily compact (see \cite{HMU}).
Let $(\frac 12, \frac 1{2^2},\frac 1{2^3}, \cdots)$ be the probability vector associated with the infinite IFS $\set{S_j}_{j=1}^\infty$. Then, there exists a unique Borel probability measure $P$ on $\D R$, actually on $J$, such that $P=\sum_{j=1}^\infty \frac 1 {2^j} P\circ S_j^{-1}$ (see \cite{RM}).
For each $j\in \D N$, we have $S_j(x)=(f_2^{(j-1)}\circ f_1)(x)$ for $x\in \D R$, where $f_1$ and $f_2$ are the two similarity mappings generating the dyadic Cantor set $C$ as defined before, and $f_2^{(j-1)}$ denotes the $j-1$ iteration of the mapping $f_2$. Therefore, by iterated application of some of the $S_j$ we can generate the mappings $f_\go:=f_{\go_1}\circ f_{\go_2}\circ \cdots \circ f_{\go_n}$ for all finite words $\go:=\go_1\go_2\cdots \go_n\in \set{1, 2}^\ast$ except of those where $\go_n=2$. Because of the contraction property, we have
\[\lim_{n\rightarrow \infty}f_{\go_1\go_2\cdots  \go_n}([0,1])=\lim_{n\to \infty}f_{w_1\ldots  w_{n-1}2}([0,1]),\]
i.e., the limit does not depend on the last letter $\go_n$. Therefore, the infinite coding mappings agree with those of the Cantor set and thus, we get $J=C$. Again, for any $j\in \D N$,
\[P(S_j([0, 1]))=\frac 1 {2^j}, \te{ and } \mu(f_2^{(j-1)}\circ f_1([0, 1]))=\frac 1{2^{j-1}} \frac 1 2=\frac 1{2^j}.\]
Thus, we see that $P=\mu$. Hence, we can say that for any $n\in \D N$, the optimal sets of $n$-means for $P$ are same as the optimal sets of $n$-means for $\mu$.  In this paper, using the infinite system of similarity mappings $\set{S_j : j\in \D N$} associated with the probability vector $(\frac 12, \frac 1{2^2}, \cdots)$, for all $n\in \D N$, we determine the optimal sets of $n$-means and the $n$th quantization errors for the infinite self-similar measure $P$. Due to infinite number of mappings the technique used in this paper is completely different from the technique used by Graf-Luschgy in \cite{GL3}. For $j\in \D N$, let $r_j\in (0, \frac 12)$. Let $\set{S_j}_{j=1}^\infty$ be a collection of contractive similarity mappings on $\D R$ such that for $x\in \D R$,
\[S_j(x)=\left\{\begin{array}{ll}
r_1 x &\te{ if } j=1,\\
r_1r_2\cdots r_j x+ 1-r_1r_2\cdots r_{j-1} & \te{ if } j\geq 2.
\end{array}
\right.
\]
Let $(p_1, p_2, \cdots)$ be a probability vector with $p_j>0$ for all $j\in \D N$. Then, for the infinite self-similar measure $P$ given by $P=\sum_{j=1}^\infty p_j P\circ S_j^{-1}$ the optimal sets of $n$-means and the $n$th quantization errors are not known yet for all infinite probability vectors $(p_1, p_2, \cdots)$ and all $0<r_j<\frac 12$. The technique developed in this paper can be used to investigate the optimal sets of $n$-means and the $n$th quantization errors for such a more general infinite self-similar measure $P$.

\section{Basic definitions, lemmas and proposition}

Let $\D N$ denote the set all natural numbers, i.e., $\D N=\set{1, 2, \cdots}$.
By a \textit{string} or a \textit{word} $\go$ over the alphabet $\D N$, we mean a finite sequence $\go:=\go_1\go_2\cdots \go_k$
of symbols from the alphabet, where $k\geq 1$, and $k$ is called the length of the word $\go$. The length of a word $\go$ is denoted by $|\go|$.  A word of length zero is called the \textit{empty word}, and is denoted by $\emptyset$. We denote the set of all words of length $k$ by $\D N^k$. By $\D N^*$ we denote the set of all words
over the alphabet $\D N$ of some finite length $k$ including the empty word $\es$. For any two words $\go:=\go_1\go_2\cdots \go_k$ and
$\tau:=\tau_1\tau_2\cdots \tau_\ell$ in $\D N^*$,
$\go\tau:=\go_1\cdots \go_k\tau_1\cdots \tau_\ell$ is the
concatenation of the words $\go$ and $\tau$. For $n\geq 1$ and $\go=\go_1\go_2\cdots\go_n\in \D N^\ast$ we define $\go^-:=\go_1\go_2\cdots\go_{n-1}$, i.e., $\go^-$ is the word obtained from the word $\go$ by deleting the last letter of $\go$. Notice that $\go^-$ is the empty word if the length of $\go$ is one. For $\go \in \D N^\ast$, by $(\go, \infty)$ it is meant the set of all words $\go^-(\go_{|\go|}+j)$, obtained by concatenation of the word $\go^-$ with the word $\go_{|\go|}+j$ for $j\in \D N$, i.e.,
\[(\go, \infty)=\set{\go^-(\go_{|\go|}+j) : j\in \D N}.\]
Write $P:=\sum_{j=1}^\infty  p_j P\circ S_j^{-1}$,  where $p_j=\frac 1{2^j}$ for all $j\in \D N$ and $\set{S_j}_{j=1}^\infty$ is the infinite collection of similitudes with similarity ratios $s_j:=\frac 1 {3^j}$ for $j\in \D N$ as defined in the previous section. Then, $P$ has support lying in the closed interval $[0, 1]$. This paper deals with this probability measure $P$. For $\go=\go_1\go_2\cdots\go_n\in \D N^n$, write
\[S_\go:=S_{\go_1}\circ\cdots \circ S_{\go_n}, \quad J_\go:=S_\go(J), \quad s_\go:=s_{\go_1}\cdots s_{\go_n}, \quad p_\go:=p_{\go_1}\cdots p_{\go_n},\]
where $J:=J_\es=[0, 1]$. If $\go$ is the empty word $\es$, then $S_\go$ represents the identity mapping on $\D R$, and $p_\go=1$. Then, for any $\go \in \D N^\ast$, we write
\[J_{(\go,\infty)}:=\mathop{\uu}\limits_{j=1}^\infty J_{{\go^-(\go_{|\go|}+j)}} \te{ and }      p_{(\go, \infty)}:=P(J_{(\go, \infty)})=\sum_{j=1}^\infty P(J_{\go^-(\go_{|\go|}+j)})=\sum_{j=1}^\infty p_{\go^-(\go_{|\go|}+j)}.\]
Notice that for any $k\in \D N$, $p_{(k, \infty)}=1-\sum_{j=1}^kp_j$, and for any word $\go \in \D N^\ast$, $p_{(\go, \infty)}=p_{\go^-}-p_\go$.

\begin{lemma} \label{lemma1}
Let $f : \mathbb R \to \mathbb R^+$ be Borel measurable and $k\in \mathbb N$. Then,
\[\int f(x) dP(x)=\sum_{\go \in \D N^k}p_\go \int (f \circ S_\go)(x) dP(x).\]
\end{lemma}

\begin{proof}
We know $P=\sum_{j=1}^\infty p_j P\circ S_j^{-1}$, and so by induction $P= \sum_{\go \in \D N^k} p_\go P\circ S_\go^{-1}$, and thus the lemma is established.
\end{proof}

\begin{lemma} \label{lemma2} Let $X$ be a random variable with probability distribution $P$. Then, the expectation $E(X)$ and the variance $V:=V(X)$ of the random variable $X$ are given by
\[E(X)=\frac 12   \text{ and } V=\frac 18.\]
\end{lemma}
\begin{proof} Using Lemma~\ref{lemma1}, we have
\begin{align*}
&E(X)=\int x dP(x)=\sum_{j=1}^\infty \frac 1{2^j} \int S_j(x) dP(x)=\sum_{j=1}^\infty \frac 1{2^j} \int\Big(\frac 1{3^j} x+1-\frac 1{3^{j-1}}\Big) dP(x)\\
&=\sum_{j=1}^\infty \Big(\frac 1{6^j} E(X)+\frac 1{2^j}(1 -\frac1{3^{j-1}})\Big)= \frac 1 5 E(X)+1-\frac 35,
\end{align*}
which implies $E(X)=\frac 12$. Now,
\begin{align*}
&E(X^2)=\int x^2 dP(x)=\sum_{j=1}^\infty \frac 1{2^j} \int\Big(\frac 1{3^j} x+1-\frac 1{3^{j-1}}\Big)^2 dP(x)\\
&=\sum_{j=1}^\infty \frac 1{2^j} \int \Big(\frac 1{9^j} x^2 + \frac 2{3^j} (1-\frac{1}{3^{j-1}}) x+(1-\frac 1{3^{j-1}})^2 \Big) dP(x).
\end{align*}
Since,
\[\sum_{j=1}^\infty \frac{1}{18^j}=\frac 1{17} \te{ and }  \sum_{j=1}^\infty \frac{1}{2^j}\int \frac 2{3^j} (1-\frac{1}{3^{j-1}})  xdP=\frac 2{85}, \te{ and } \sum_{j=1}^\infty \frac{1}{2^j} (1-\frac 1{3^{j-1}})^2=\frac{28}{85},\]
we have
$E(X^2)=\frac 1{17}E(X^2)+\frac 2{85}+\frac{28}{85}$  which yields $E(X^2)=\frac 3{8}$. Thus,
\[V=E(X^2)-\left(E(X)\right)^2 =\frac 38 -\frac 14=\frac 18,\]
which is the lemma.
\end{proof}

\begin{lemma} \label{lemma3}
For any $k\geq 1$, we have
\[E(X | X \in J_k\uu J_{k+1}\uu \cdots)=1-\frac 12 \frac{1}{3^{k-1}}.\]
\end{lemma}
\begin{proof} By the definition of conditional expectation, we have
\begin{align*}
&E(X | X \in J_k\uu J_{k+1}\uu \cdots)=\frac {1}{\sum_{j=k}^\infty \frac 1{2^j}} \Big(\sum_{j=k}^\infty \frac 1 {2^j} S_j(\frac 12)\Big)=2^{k-1} \sum_{j=k}^\infty \frac 1 {2^j}(1-\frac 5 2 \frac 1{3^{j}})=1-\frac 12 \frac 1{3^{k-1}},
\end{align*}
which is the lemma.
\end{proof}

Now, the following remarks are in order.

\begin{remark} \label{remark1} For $k\in \D N$, we have
$S_k(\frac 12)=\frac 1{3^k}\frac 12 +1-\frac 1{3^{k-1}}=1-\frac 5 2 \frac 1{3^{k}}.$ Thus, by Lemma~\ref{lemma3}, for $k\in \D N$,
\begin{equation*}  E(X | X \in J_k\uu J_{k+1}\uu \cdots)=S_k(\frac 12) +\frac 1{3^k}=\frac 12(S_k(1)+S_{k+1}(0)).
\end{equation*}   For any $x_0 \in \D R$, we have
$\int(x-x_0)^2 dP(x) =V(X)+(x_0-E(X))^2$, and so, the optimal set of one-mean is the expected value and the corresponding quantization error is the variance $V$ of the random variable $X$. For $\go \in \D N^k$, $k\geq 1$, using Lemma~\ref{lemma1}, we have
\begin{align*}
&E(X : X \in J_\go) =\frac{1}{P(J_\go)} \int_{J_\go} xdP(x)=\int_{J_\go} x  dP\circ S_\go^{-1}(x)=\int S_\go(x)  dP(x)=E(S_\go(X)).
\end{align*}
Since $S_j$ are similitudes, we have $E(S_j(X))=S_j(E(X))$ for $j\in \D N$, and so by induction, $E(S_\go(X))=S_\go(E(X))$
 for $\go\in \D N^k$, $k\geq 1$.
\end{remark}

\begin{remark}

For words $\gb, \gg, \cdots, \gd$ in $\D N^\ast$, by $a(\gb, \gg, \cdots, \gd)$ we denote the conditional expectation of the random variable $X$ given $J_\gb\uu J_\gg \uu\cdots \uu J_\gd,$ i.e.,
\begin{equation*}  a(\gb, \gg, \cdots, \gd)=E(X|X\in J_\gb \uu J_\gg \uu \cdots \uu J_\gd)=\frac{1}{P(J_\gb\uu \cdots \uu J_\gd)}\int_{J_\gb\uu \cdots \uu J_\gd} x dP(x).
\end{equation*}
Thus by Remark~\ref{remark1}, for $\go\in \D N^\ast$, we have
\begin{align} \label{eq1} \left\{ \begin{array}{ll} a(\go)=S_\go(E(X))=S_\go(\frac 12), \te{ and  } & \\
 a(\go, \infty)=E(X|X\in J_{\go^-(\go_{|\go|}+1)}\uu J_{\go^-(\go_{|\go|}+2)}\uu \cdots)=S_{\go^-(\go_{|\go|}+1)}(\frac 12)+s_{\go^-(\go_{|\go|}+1)}.&
 \end{array} \right.
 \end{align}

Moreover, for any  $\go \in \D N^\ast$   and $j\geq 1$, since $p_{\go^-(\go_{|\go|}+j)}=p_{\go^-} p_{\go_{|\go|}+j}=p_{\go^-} p_{\go_{|\go|}}p_j=p_\go p_j=p_{\go j}$,
and similarly $s_{\go^-(\go_{|\go|}+j)}=s_\go s_j=s_{\go j}$, for any $x_0\in \D R$, we have
\begin{align} \label{eq2} \left\{ \begin{array}{ll}\int_{J_\go}(x-x_0)^2 dP(x)=p_\go\int (x -x_0)^2 dP\circ S_\go^{-1}(x)=p_\go  \Big(s_\go^2V+(S_\go(\frac 12)-x_0)^2\Big),\te{ and } &\\
\int_{J_{(\go, \infty)}} (x -x_0)^2  dP(x)=\sum_{j=1}^\infty p_{\go j} \Big(s_{\go j}^2 V + (S_{\go^-(\go_{|\go|}+j)}(\frac 12)-x_0)^2\Big).
\end{array} \right.
 \end{align}
In the sequel, for  $\go \in \D N^k$, $k\geq 1$, by $E(a(\go))$, we mean the error contributed by $a(\go)$ in the region $J_\go$; and similarly, by $E(a(\go,\infty))$, it is meant the error contributed by $a(\go,\infty)$ in the region $J_{(\go, \infty)}$. We apologize for any abuse in notation. Thus, we have
\begin{align}\label{eq43}
E(a(\go)):=\int_{J_{\go}}(x-a(\go))^2 dP(x) \te{ and } E(a(\go, \infty)):=\int_{J_{(\go, \infty)}}(x-a(\go, \infty))^2 dP(x).
\end{align}
\end{remark}

We now prove the following lemma.
\begin{lemma} \label{lemma4} Let $\go \in \D N^\ast$. Let $E(a(\go))$ and $E(a(\go, \infty))$ be defined by \eqref{eq43}. Then,
\[E(a(\go))=E(a(\go, \infty))=p_\go s_\go^2 V.\]
\end{lemma}

\begin{proof} In the first equation of \eqref{eq2} put $x_0=a(\go)$, and then $E(a(\go))=p_\go s_\go^2 V$. In the second equation of \eqref{eq2}, put $x_0=a(\go, \infty)$, and then
\begin{align} \label{eq36} & E(a(\go, \infty))=\sum_{j=1}^\infty p_{\go j}\Big(s_{\go j}^2 V +(S_{\go^-(\go_{|\go|}+j)}(\frac 12)-a(\go, \infty))^2\Big).
\end{align}
Putting the values of $a(\go, \infty)$ from \eqref{eq1}, we have
\begin{align*}
&S_{\go^-(\go_{|\go|}+j)}(\frac 12)-a(\go, \infty)=S_{\go^-(\go_{|\go|}+j)}(\frac 12)-S_{\go^-(\go_{|\go|}+1)}(\frac 12)-s_{\go^-(\go_{|\go|}+1)}\\
&=s_{\go^-} \Big(S_{\go_{|\go|}+j}(\frac 12)-S_{\go_{|\go|}+1}(\frac 12)-s_{\go^-(\go_{|\go|}+1)}\Big)\\
&=s_{\go^-}\Big(\frac{1}{3^{\go_{|\go|+j}}} \frac 12-\frac{1}{3^{\go_{|\go|+j-1}}}-\frac{1}{2^{\go_{|\go|+1}}} \frac 12+\frac{1}{2^{\go_{|\go|}}}- s_{\go^-(\go_{|\go|}+1)}\Big)\\
&=s_{\go}\Big(\frac{1}{3^{j}} \frac 12-\frac{3}{3^{j}}- \frac 16+1-\frac 13\Big)=s_\go\Big(\frac{1}{2}-\frac{5}{2} \frac 1{ 3^j}\Big).
\end{align*}
Hence, \eqref{eq36} implies
$E(a(\go, \infty)) =p_\go s_\go^2 \sum_{j=1}^\infty \frac {1}{2^j}\Big (\frac 1{9^j} V+(\frac 12-\frac 52 \frac 1{3^j})^2\Big)=p_\go s_\go^2 V.
$
Thus, the proof of the lemma is complete.
\end{proof}

\begin{remark} By \eqref{eq2} and Lemma~\ref{lemma4}, for any $x_0 \in \D R$, we have
\begin{align} \label{eq222} \left\{ \begin{array}{ll}\int_{J_\go}(x-x_0)^2 dP(x)=E(a(\go))+(x_0-a(\go))^2 p_\go,\te{ and } &\\
\int_{J_{(\go, \infty)}} (x -x_0)^2  dP(x)=E(a(\go))+(x_0-a(\go, \infty))^2(p_{\go^-}-p_{\go}).
\end{array} \right.
 \end{align}
Notice that by \eqref{eq1}, we have $a(\go, \infty) =a(\go^-(\go_{|\go|}+1))+s_{\go^-(\go_{|\go|}+1)}$. The expressions \eqref{eq1} and \eqref{eq222}  are useful to obtain the optimal sets and the corresponding quantization errors with respect to the probability distribution $P$.
\end{remark}

The following lemma is useful.

\begin{lemma}\label{lemma5} For any two words $\go, \gt \in \D N^\ast$, if $p_\go=p_\gt$, then
\[\int_{J_\go} (x-a(\go))^2  dP(x)=\int_{J_\gt}(x-a(\gt))^2 dP(x).\]
\end{lemma}
\begin{proof} Let $\go, \gt \in \D N^\ast$. Let $\go=\go_1\go_2\cdots \go_k$ and $\gt=\gt_1\gt_2\cdots\gt_m$ for some $k, m\in \D N$.
Then, $p_\go=p_\gt$ implies $\go_1+\go_2+\cdots +\go_k=\gt_1+\gt_2+\cdots+\gt_m$, and so $s_\go=s_\gt$. Thus,
\[\int_{J_\go} (x-a(\go))^2  dP(x)=p_\go s_\go^2 V=p_\gt s_\gt^2 V=\int_{J_\gt}(x-a(\gt))^2 dP(x),\]
which is the lemma.
\end{proof}

\begin{defi}\label{defi1}
For $n\in \D N$ with $n\geq 2$ let $\ell(n)$ be the unique natural number with $2^{\ell(n)} \leq n<2^{\ell(n)+1}$. Write
 \[\ga(\ell(n)):=\{a(\go)  : \go \in \D N^\ast\te{ and } p_\go=\frac{1}{2^{\ell(n)}}\}\uu \{a(\go, \infty)  : \go \in \D N^\ast\te{ and } p_\go=\frac{1}{2^{\ell(n)}}\}.\]
For $I\sci\ga(\ell(n))$ with card$(I)=n-2^{\ell(n)}$, write
\begin{align*} \ga_n(I):&=(\ga(\ell(n)) \setminus I)\uu \set{a (\go 1) : a(\go) \in I} \uu \set {a(\go1, \infty) : a(\go) \in I}\\
& \qquad \uu \set{a (\go^-(\go_{|\go|}+ 1)) : a(\go, \infty) \in I} \uu \set {a(\go^-(\go_{|\go|}+ 1), \infty) : a(\go, \infty) \in I}.
\end{align*}

\end{defi}

\begin{remark} In Definition~\ref{defi1}, if $n=2^{\ell(n)}$, then $I=\es$, and so, $\ga_n(I)= \ga(\ell(n))$.

\end{remark}

Using Definition~\ref{defi1}, we now give a few examples.

\begin{example} Let $n=3$. Then, $\ell(n)=1$, $\ga(1)=\set{a(1), a(1, \infty)}=\set{\frac 16, \frac 56}$,  $\te{card}(I)=1$.
If $I=\set{a(1)}$, then
\[\ga_3(I)=\set{a(11), a(11, \infty), a(1, \infty)}=\set{\frac 1{18}, \frac 5{18}, \frac 56}.\]
If $I=\set{a(1, \infty)}$, then,
\[\ga_3(I)=\set{a(1), a(2), a(2, \infty)}=\set{\frac 1{6}, \frac {13}{18}, \frac{17}{18}}.\]
\end{example}

\begin{example} Let $n=4$. Then, $\ell(n)=2$, $I=\es$, and so
\[\ga_4(I)=\ga(2)=\set{a(11), a(11, \infty), a(2), a(2, \infty)}=\set{\frac 1{18}, \frac 5{18}, \frac {13}{18}, \frac{17}{18}}.\]

\end{example}

\begin{example} Let $n=5$. Then, $\ell(n)=2$, $\ga(2)=\set{a(11), a(11, \infty), a(2), a(2, \infty)}$, $I\sci \ga(2)$ with card$(I)=1$. If $I=\set{a(11)}$, then
\[\ga_5(I)=\set{a(111), a(111, \infty), a(11, \infty), a(2), a(2, \infty)}=\set{\frac 1{54}, \frac {5}{54}, \frac 5{18}, \frac {13}{18}, \frac{17}{18}}.\]
 If $I=\set{a(2)}$, then
\[\ga_5(I)=\set{a(11), a(11, \infty), a(21), a(21, \infty), a(2, \infty)}=\set{\frac 1{18}, \frac {5}{18}, \frac {37}{54}, \frac {41}{54}, \frac{17}{18}}.\]
If $I=\set{a(11, \infty)}$, then
\[\ga_5(I)=\set{a(11), a(12), a(12, \infty), a(2), a(2, \infty)}=\set{\frac 1{18}, \frac {13}{54}, \frac {17}{54}, \frac {13}{18}, \frac{17}{18}}.\]
If $I=\set{a(2, \infty)}$, then
\[\ga_5(I)=\set{a(11), a(11, \infty), a(2), a(3), a(3, \infty)}=\set{\frac 1{18}, \frac {5}{18}, \frac {13}{18}, \frac {49}{54}, \frac{53}{54}}.\]
\end{example}
Let us now prove the following proposition.
\begin{prop} \label{prop1}
Let $\ga_n(I)$ be the set as defined by Definition~\ref{defi1}. Then
\[\int \min_{a \in \ga_n(I)}(x-a)^2 dP(x)= \frac 1{18^{\ell{(n)}}}  \frac 1 8 \Big(2^{\ell(n)+1}-n +\frac 19 (n-2^{\ell(n)})\Big).\]
\end{prop}

\begin{proof}
Using the definition of $\ga_n(I)$, we have
\begin{align*}
&\int \min_{a \in \ga_n(I)}(x-a)^2 dP(x)\\
&=\sum_{a(\go)  \in \ga(\ell(n))\setminus I}\int_{J_\go}(x-a(\go))^2 dP(x)+ \sum_{a(\go, \infty)  \in \ga(\ell(n))\setminus I}\int_{J_(\go, \infty)}(x-a(\go, \infty))^2 dP(x)\\
&+\sum_{a(\go) \in I} \Big ( \int_{J_{\go1}} (x-a(\go1))^2 dP(x)+ \int_{J_{(\go1, \infty)}} (x-a(\go1, \infty))^2 dP(x)\Big)\\
& +\sum_{a(\go, \infty) \in I} \Big ( \mathop{\int}\limits_{J_{\go^-(\go_{|\go|}+1)}} (x-a(\go^-(\go_{|\go|}+1)))^2 dP(x)+ \mathop{\int}\limits_{J_{(\go^-(\go_{|\go|}+1), \infty)}} (x-a(\go^-(\go_{|\go|}+1), \infty))^2 dP(x)\Big).
\end{align*}
Now, using Lemma~\ref{lemma4}, we have
\begin{align*}
&\sum_{a(\go)  \in \ga(\ell(n))\setminus I}\int_{J_\go}(x-a(\go))^2 dP(x)+ \sum_{a(\go, \infty)  \in \ga(\ell(n))\setminus I}\int_{J_(\go, \infty)}(x-a(\go, \infty))^2 dP(x)\\
&=\sum_{a(\go)  \in \ga(\ell(n))\setminus I}p_\go s_\go^2 V+ \sum_{a(\go, \infty)  \in \ga(\ell(n))\setminus I}p_\go s_\go^2 V\\
&=\frac 1{18^{\ell{(n)}}}  \frac 1 8 \te{ card}(\ga(\ell(n))\setminus I)=\frac 1{18^{\ell{(n)}}}  \frac 1 8 ( 2^{\ell(n)+1}-n).
\end{align*}
Again, by Lemma~\ref{lemma4}, we have
\begin{align*}
&\sum_{a(\go) \in I} \Big ( \int_{J_{\go1}} (x-a(\go1))^2 dP(x)+ \int_{J_{(\go1, \infty)}} (x-a(\go1, \infty))^2 dP(x)\Big)=2 p_1s_1^2 V \sum_{a(\go) \in I} p_\go s_\go^2,
\end{align*}
and
\begin{align*}
&\sum_{a(\go, \infty) \in I} \Big ( \mathop{\int}\limits_{J_{\go^-(\go_{|\go|}+1)}} (x-a(\go^-(\go_{|\go|}+1)))^2 dP(x)+ \mathop{\int}\limits_{J_{(\go^-(\go_{|\go|}+1), \infty)}} (x-a(\go^-(\go_{|\go|}+1), \infty))^2 dP(x)\Big)\\
&=2 p_1s_1^2 V \sum_{a(\go, \infty) \in I} p_\go s_\go^2.
\end{align*}
Combining all these,
\begin{align*}
&\int \min_{a \in \ga_n(I)}(x-a)^2 dP(x)=\frac 1{18^{\ell{(n)}}}  \frac 1 8 ( 2^{\ell(n)+1}-n)+2 p_1s_1^2 V\Big(\sum_{a(\go) \in I} p_\go s_\go^2+\sum_{a(\go, \infty) \in I} p_\go s_\go^2\Big)\\
&=\frac 1{18^{\ell{(n)}}}  \frac 1 8 ( 2^{\ell(n)+1}-n)+\frac 1 9 \frac 1 8 \frac 1 {18^{\ell(n)}} \te{card}(I)=\frac 1{18^{\ell{(n)}}}  \frac 1 8\Big(2^{\ell(n)+1}-n+\frac 1 9(n-2^{\ell(n)})\Big),
\end{align*}
which is the proposition.
\end{proof}

\begin{cor} \label{cor21}
Let $V_n$ be the $n$th quantization error for every $n\geq 1$. Then,
\[V_n\leq \frac 1{18^{\ell{(n)}}}  \frac 1 8\Big(2^{\ell(n)+1}-n+\frac 1 9(n-2^{\ell(n)})\Big).\]
\end{cor}

In the next section, Theorem~\ref{Th1} gives the optimal sets of $n$-means and the $n$th quantization errors for all $n\geq 2$.

\section{Optimal sets of $n$-means for all $n\geq 2$}
In this section, first we give some basic lemmas and propositions that we need to state and prove Theorem~\ref{Th1} which gives the main result of the paper. To prove the lemmas and propositions, we will frequently use the formulas given by the expressions \eqref{eq1} and \eqref{eq222}.

\begin{lemma}\label{lemma11} Let $\ga:=\{a_1, a_2\}$ be an optimal set of two-means, $a_1<a_2$. Then, $a_1=a(1)=\frac 16$, $a_2=a(1, \infty)=\frac 56$ and the corresponding quantization error is $V_2=\frac{1}{72}=0.0138889$.
\end{lemma}
\begin{proof}
by Corollary~\ref{cor21},  $V_2\leq \frac 1{72}=0.0138889$.
 Let $\ga=\{a_1, a_2\}$ be an optimal set of two-means, $a_1<a_2$. Since $a_1$ and $a_2$ are the centroids of their own Voronoi regions, we have $0\leq a_1<a_2\leq 1$. If $a_1\geq \frac 13$, then
\begin{align*}
\frac 1{72}\geq V_2 \geq \int_{J_1}(x-\frac 13)^2 dP=\frac{1}{48}>\frac 1{72}>V_2,
\end{align*}
which is a contradiction, and so $a_1< \frac 13$. If $a_2\leq \frac 23$, then
\begin{align*}
&\frac 1{72}\geq V_2 \geq \int_{J_{(1, \infty)}}(x-\frac 23)^2 dP>\int_{J_2\uu J_3\uu J_4}(x-\frac 23)^2 dP=\frac{3959}{279936}=0.0141425>V_2,
\end{align*}
which leads to a contradiction. Thus, $\frac 23<a_2$. Since $0\leq a_1\leq \frac 13 <\frac 23\leq  a_2\leq 1$, we have $\frac 13 \leq \frac{a_1+a_2}{2}\leq \frac 23$, and so $J_1 \ci M(a_1|\ga)$ and $J_{(1, \infty)} \ci M(a_2|\ga)$. Thus,
\[\int\min_{a\in \ga}(x-a)^2 dP=\int_{J_1}(x-a_1)^2 dP+\int_{J_{(1, \infty)}}(x-a_2)^2 dP,\]
which is minimum when $a_1=a(1)=S_1(\frac 12)=\frac 16 \te{ and } a_2=a(1,\infty)=S_2(\frac 12) +\frac 1{3^2}=\frac 56,$
and the corresponding quantization error is $V_2=\frac 1{72}$. Hence, the proof of the lemma is complete.
\end{proof}


\begin{prop} \label{prop211}
Let $\ga_n$ be an optimal set of $n$-means for $n\geq 2$. Then, $\ga_n\ii J_1\neq \es$ and $\ga_n\ii [\frac 23, 1]\neq\es$. Moreover, the Voronoi region of any point in $\ga_n\ii J_1$ does not contain any point from $[\frac 23, 1]$ and the Voronoi region of any point in $\ga_n\ii [\frac 23, 1]$ does not contain any point from $J_1$.
\end{prop}

\begin{proof} By Lemma~\ref{lemma11}, the proposition is true for $n=2$. We now show that the proposition is true for all $n\geq 3$. Consider the set of three points $\gb$ given by $\gb:=\set{a(11), a(11, \infty), a(1, \infty)}$. Then, the distortion error is
\[\int\min_{a\in \gb} (x-a)^2dP=\mathop{\int}\limits_{J_{11}}(x-a(11))^2 dP+\mathop{\int}\limits_{J_{(11, \infty)}}(x-a(11, \infty))^2 dP+\mathop{\int}\limits_{J_{(1, \infty)}}(x-a(1, \infty))^2 dP=\frac{5}{648}.\]
Since $V_n$ is the quantization error for $n$-means for all $n\geq 3$, we have $V_n\leq V_3\leq \frac{5}{648}=0.00771605$. Let $\ga_n:=\set{a_1<a_2<\cdots<a_n}$ be an optimal set of $n$-means for $n\geq 3$. Since the optimal quantizers  are the centroids of their own Voronoi regions, we have $0\leq a_1<a_2\cdots<a_n\leq 1$.
Proceeding in the similar way as Lemma~\ref{lemma11}, it can be shown that $a_1<\frac 13$ and $\frac 23<a_n$ yielding the fact that $\ga_n\ii J_1\neq \es$ and $\ga_n\ii [\frac 23, 1]\neq \es$. Let $j=\max\set{ i : a_i\leq \frac 13}$. Then, $a_j\leq \frac 13$. Suppose that the Voronoi region of $a_j$ contains points from $[\frac 23, 1]$. Then, we must have $\frac 12(a_j+a_{j+1})>\frac 23$ implying $a_{j+1}>\frac 43-a_j\geq \frac 43-\frac 13=1$, which gives a contradiction. Hence, the Voronoi region of any point in $\ga_n\ii J_1$ does not contain any point from $J_{(1, \infty)}$. Similarly, we can show that the Voronoi region of any point in $\ga_n\ii [\frac 23, 1]$ does not contain any point from $J_1$. Thus, the proof of the proposition is complete.
\end{proof}

We need the following two lemmas to prove Lemma~\ref{lemma12}.

\begin{lemma} \label{lemma120}
Let $V(P, J_1, \set{a, b})$ be the quantization error due to the points $a$ and $b$ on the set $J_1$, where $0\leq  a<b$ and $b=\frac 13$. Then, $a=a(11)$ and
\[V(P, J_1, \set{a, b})=\int_{J_{11}}(x-a(11))^2 dP+\int_{J_{(11, \infty)}}(x-\frac 13)^2dP=\frac{1}{648}.\]
\end{lemma}

\begin{proof}  Consider the set $\set{a(11), \frac 13}$. Then, as $S_{11}(1)=\frac 19<\frac 12(a(11)+\frac 13)=\frac{7}{36}<S_{12}(0)=\frac 29$, and $V(P, J_1, \set{a, b})$ is the quantization error due to the points $a$ and $b$ on the set $J_1$, we have
\[V(P, J_1, \set{a, b})\leq \int_{J_{11}}(x-a(11))^2 dP+\int_{J_{(11, \infty)}}(x-\frac 13)^2dP=\frac{1}{2592}+\frac{1}{864}=\frac{1}{648}=0.00154321. \]
If $\frac{1}{8}\leq a$, then
\[V(P, J_1, \set{a, b})\geq \int_{J_{11}}(x-\frac 18)^2 dP=\frac{11}{6912}=0.00159144>V(P, J_1, \set{a, b}),\]
which is a contradiction, and so we can assume that $a<\frac 18$. If the Voronoi region of $b$ contains points from $J_{1}$, we must have
$\frac 12(a+b)<\frac 19$ implying $a<\frac 29-b= \frac 29-\frac13=-\frac 19$, which leads to a contradiction. So, we can assume that the Voronoi region of $b$ does not contain any point from $J_{11}$ yielding $a\geq a(11)=\frac{1}{18}$. If the Voronoi region of $a$ contains points from $[\frac 29, \frac 13]$, we must have
$\frac 12(a+\frac 13)>\frac 29$ implying $a>\frac 49-\frac 13=\frac 19$, and so $\frac 19<a\leq \frac 18$. But, then $\frac 12(\frac 18+\frac 13)=\frac{11}{48}<S_{121}(1)$ yielding
\[V(P, J_1, \set{a, b})>\int_{J_{11}}(x-\frac 19)^2 dP+\int_{J_{122}\uu J_{123}\uu J_{124}\uu J_{13}}(x-\frac 13)^2=\frac{2577311}{1632586752}=0.00157867,\]
and so, $V(P, J_1, \set{a, b})>0.00157867>V(P, J_1, \set{a, b})$, which gives a contradiction. Hence, the Voronoi region of $a$ does not contain any point from $[\frac 29, \frac 13]$ yielding $a\leq a(11)$. Again, we have seen $a\geq a(11)$. Thus, $a=a(11)$ and
\[V(P, J_1, \set{a, b})=\int_{J_{11}}(x-a(11))^2 dP+\int_{J_{(11, \infty)}}(x-\frac 13)^2dP=\frac{1}{648},\]
which is the lemma.
\end{proof}

Proceeding in the similar way as Lemma~\ref{lemma120}, the following lemma can be proved.
\begin{lemma} \label{lemma1200}
Let $V(P, J_{(1, \infty)}, \set{a, b})$ be the quantization error due to the points $a$ and $b$ on the set $J_{(1, \infty)}$, where $a=\frac 23$ and $\frac23<b\leq 1$. Then, $b=a(2, \infty)$ and
\[V(P, J_{(1, \infty)}, \set{a, b})=\int_{J_{2}}(x-\frac 23)^2 dP+\int_{J_{(2, \infty)}}(x-a(2, \infty))^2dP=\frac{1}{648}.\]
\end{lemma}

\begin{lemma} \label{lemma12}
Let $\ga$ be an optimal set of three-means. Then, $\ga=\set{a(11), a(11, \infty), a(1, \infty)}=\set{\frac 1{18}, \frac 5{18}, \frac 56}$, or $\ga=\set{a(1), a(2), a(2, \infty)}=\set{\frac 1{6}, \frac {13}{18}, \frac{17}{18}}$ with quantization error $V_3=\frac{5}{648}=0.00771605$.
\end{lemma}
\begin{proof} As shown in the proof of Proposition~\ref{prop211}, if $V_3$ is the quantization error for three-means, we have $V_3\leq \frac{5}{648}=0.00771605$. Let $\ga$ be an optimal set of three-means with $\ga=\set{a_1, a_2, a_3}$, where $a_1<a_2<a_3$. By Propositoin~\ref{prop211}, we have $0\leq a_1<\frac 13$ and $\frac 23< a_3\leq 1$. We now show that $\ga_3$ does not contain any point from the open interval $(\frac 13, \frac 23)$. For the sake of contradiction, assume that $a_2 \in (\frac 13, \frac 23)$. The following two cases can arise:

Case 1: $a_2 \in [\frac 12, \frac 23)$.

Then, $\frac 12(a_1+a_2)<\frac 13$ implying $a_1<\frac 23-a_2\leq \frac 23-\frac 12=\frac 16=a(1)$, otherwise, the quantization error can be strictly reduced by moving the point $a_2$ to $\frac 23$. Thus, by Lemma~\ref{lemma1200}, we have
\begin{align*}
V_3\geq \int_{J_{1}}(x-\frac 16)^2 dP+\frac{1}{648} =\frac{11}{1296}=0.00848765>V_3,
\end{align*}
which leads to a contradiction.

Case 2: $a_2 \in (\frac 13, \frac 12]$.

Then, $\frac 12(a_2+a_3)>\frac 23$ implying $a_3>\frac 43-a_2\geq \frac 43-\frac 12=\frac 56=a(1, \infty)$. Then, by Lemma~\ref{lemma120}, we have
\begin{align*}
V_3\geq \frac{1}{648}+\int_{J_{(1, \infty)}}(x-a(1, \infty))^2 dP=\frac{11}{1296}=0.00848765>V_3,
\end{align*}
which gives a contradiction.

Thus, by Case 1 and Case 2, we have $a_2 \not \in (\frac 13, \frac 23)$, i.e., either $a_2 \in [0, \frac 13]$ or $a_2 \in [\frac 23, 1]$. Let us first assume $a_2 \in [0, \frac 13]=J_1$. Set $\ga_1:=\set{a_1, a_2}$ and $\ga_2:=\set{a_3}$. Since $\ga=\ga_1\uu \ga_2$, by Lemma~\ref{lemma1}, we deduce
\[V_3=\int_{J_1}\min_{a \in \ga_1}(x-a)^2 dP+\int_{J_{(1, \infty)}}(x-a_3)^2 dP=\frac 1{18} \int\min_{a \in 3\ga_1}(x-a)^2 dP+\int_{J_{(1, \infty)}}(x-a_3)^2 dP.\]
We now show that $S_1^{-1}(\ga_1)$ is an optimal set of two-means. If $S_1^{-1}(\ga_1):=3\ga_1$ is not an optimal set of two-means, then we can find a set $\gb \sci \D R$ with $\te{card}(\gb)=2$ such that $\int\mathop{\min}\limits_{b\in \gb} (x-b)^2 dP<\int \mathop{\min}\limits_{a\in \ga_1}(x-3a)^2 dP$. But, then $(\frac 1 3 \gb) \uu \ga_2$ is a set of cardinality three with $\int{\min}_{a\in \frac 13 \gb\uu \ga_2}(x-a)^2 dP<\int{\min}_{a\in \ga} (x-a)^2 dP$, which contradicts the optimality of $\ga$. Thus, $S_1^{-1}(\ga_1)$ is an optimal set of two-means, i.e., $S_1^{-1}(\ga_1)=\set{a(1), a(1, \infty)}$ which gives $\ga_1=\set{a(11), a(11, \infty)}$. Again, $V_3$ being the quantization error, we must have $a_3=a(1, \infty)$. Thus, under the assumption $a_2 \in [0, \frac 13]=J_1$, we have $\ga=\set{a(11), a(11, \infty), a(1, \infty)}$, and then using \eqref{eq222}, we have $V_3=\frac{5}{648}$.
Let us now assume $\frac 23\leq a_2$.  Set $\gb:=\set{a_2, a_3}$. Then,
\[V_3=\int_{J_1}(x-a(1))^2 dP+\int_{J_{(1, \infty)}}\min_{b \in \gb} (x-b)^2 dP=\frac{1}{144}+\int_{J_{(1, \infty)}}\min_{b \in \gb} (x-b)^2 dP.\]
We show that $a_2<S_2(1)=\frac 79$ and $S_3(0)=\frac 89<a_3$. If $a_2\geq \frac 79$, then
\[V_3\geq \frac{1}{144}+\int_{J_2}(x-\frac 79)^2dP=\frac{7}{864}=0.00810185>V_3,\]
which leads to a contradiction. If $a_3\leq \frac 89=S_3(0)$, then,
\[V_3\geq \frac{1}{144}+\int_{J_3\uu J_4\uu J_5}(x-\frac 89)^2dP=\frac{38951}{5038848}=0.00773014>V_3,\]
which give a contradiction. Thus, $a_2<S_2(1)=\frac 79$ and $S_3(0)=\frac 89<a_3$ yielding
\[\int_{J_{(1, \infty)}}\min_{b\in \gb}(x-b)^2 dP=\int_{J_{2}}(x-a_2)^2 dP+\int_{J_{(2, \infty)}}(x-a_3)^2 dP,\]
which is minimum when $a_2=a(2)\te{ and } a_3=a(2,\infty)$. Hence, under the assumption $a_2\in [\frac 23, 1]$, we obtain $\ga=\set{a(1), a(2), a(2, \infty)}$ and $V_3=\frac{5}{648}$. Thus, the proof of the lemma is complete.
\end{proof}

\begin{prop} \label{prop212}
Let $n\geq 2$ and $\ga_n$ be an optimal set of $n$-means. Then, $\ga_n$ does not contain any point from the open interval $(\frac 13, \frac 23)$.
\end{prop}
\begin{proof} By Lemma~\ref{lemma11} and Lemma~\ref{lemma12}, the proposition is true for $n=2$ and $n=3$. Let us now prove that the proposition is true for all $n\geq 4$. Consider the set of four points $\gb:=\set{a(11), a(11, \infty), a(2), a(2, \infty)}$. Then, by Lemma~\ref{lemma4}, we have the distortion error as
\begin{align*}
\int \min_{a\in\gb}(x-a)^2 dP&=2\Big(E(11)+E(2)\Big )=\frac{1}{648}.
\end{align*}
Since $V_n$ is the quantization error for $n$-means for $n\geq 4$, we have $V_n\leq V_4\leq \frac{1}{648}=0.00154321$. Let $\ga_n:=\set{a_1<a_2<\cdots<a_n}$ be an optimal set of $n$-means for $n\geq 4$. Since the optimal quantizers  are the centroids of their own Voronoi regions, we have $0\leq a_1<a_2\cdots<a_n\leq 1$. Let $j:=\max\set{i : a_i\leq \frac 13}$. Then, $a_j\leq \frac 13$. Proposition~\ref{prop211} implies that $2\leq j\leq n-1$. We need to show that $\frac 23\leq a_{j+1}$. Suppose that $a_{j+1}\in (\frac 13, \frac 23)$.  Then, either $a_{j+1} \in [\frac 12, \frac 23)$, or $a_j \in (\frac 13, \frac 12]$. First, assume that $a_{j+1} \in [\frac 12, \frac 23)$.
Then, $\frac 12(a_j+a_{j+1})<\frac 13$ implying $a_j<\frac 23-a_{j+1}\leq \frac 23-\frac 12=\frac 16<\frac 29=S_{12}(0)$, and so
\begin{align*}
V_n\geq \int_{J_{12}\uu J_{13}}(x-\frac 16)^2 dP=\frac{521}{279936}=0.00186114>V_n,
\end{align*}
which leads to a contradiction. Next, assume that $a_{j+1} \in (\frac 13, \frac 12]$. Then, $\frac 12(a_{j+1}+a_{j+2})>\frac 23$ implying $a_{j+2}>\frac 43-a_{j+1}=\frac 43-\frac 12=\frac 56>S_2(1)$, and so,
\begin{align*}
V_n\geq \int_{J_{2}}(x-\frac 56)^2 dP=\frac{1}{288}=0.00347222>V_n,
\end{align*}
which gives another contradiction. Hence, $\frac 23\leq a_{j+1}$, which completes the proof of the proposition.
\end{proof}

\begin{lemma} \label{lemma213}
Let $\ga_n$ be an optimal set of $n$-means for $n\geq 4$. Then, $\te{card}(\ga_n\ii J_1)\geq 2$ and $\te{card}(\ga_n\ii [S_2(0), 1])\geq 2$.
\end{lemma}
\begin{proof} As shown in the proof of Proposition~\ref{prop212}, since $V_n$ is the quantization error for $n$-means for $n\geq 4$, we have $V_n\leq V_4\leq \frac{1}{648}=0.00154321$. By Proposition~\ref{prop211}, we have $\te{card}(\ga_n\ii J_1)\geq 1$ and $\te{card}(\ga_n\ii [S_2(0), 1])\geq 1$. First, we show that $\te{card}(\ga_n\ii [S_2(0), 1])\geq 2$. Suppose that $\te{card}(\ga_n\ii [S_2(0), 1])=1$. Then,
\[V_n\geq \int_{J_{(1, \infty)}}(x-a(1, \infty))^2 dP=\frac{1}{144}=0.00694444>V_n,\]
which leads to a contradiction. So, we can assume that $\te{card}(\ga_n\ii [S_2(0), 1])\geq 2$ for $n\geq 4$. Next, suppose that $\te{card}(\ga_n\ii J_1)=1$. Then,
\[V_n\geq \int_{J_1}(x-a(1))^2 dP=\frac{1}{144}=0.00694444>V_n,\]
which leads to another contradiction. Thus, the lemma is established.
\end{proof}

\begin{prop} \label{prop214} Let $\ga_n$ be an optimal set of $n$-means for $P$ such that $\te{card}(\ga_n\ii [S_{k+1}(0), 1])\geq 2$ for some $k\in \D N$ and $n\in \D N$. Then, $\ga_n\ii J_{k+1}\neq \es$, $\ga_n\ii [S_{k+2}(0), 1]\neq \es$, and $\ga_n$ does not contain any point from the open interval $(S_{k+1}(1),  S_{k+2}(0))$.
Moreover, the Voronoi region of any point in $\ga_n\ii J_{k+1}$ does not contain any point from $[S_{k+2}(0), 1]$ and the Voronoi region of any point in $\ga_n\ii [S_{k+2}(0), 1]$ does not contain any point from $J_{k+1}$.
\end{prop}
\begin{proof} To prove the proposition it is enough to prove it for $k=1$, and then inductively the proposition will follow for all $k\geq 2$. Fix $k=1$.
Suppose that $\te{card}(\ga_n \ii [S_2(0), 1])\geq 2$. By Lemma~\ref{lemma12}, it is clear that the proposition is true for $n=3$. We now prove  that the proposition is true for $n=4$.
Let $\ga_4:=\set{a_1, a_2, a_3, a_4}$ be an optimal set of four-means, such that $0< a_1<a_2<a_3<a_4< 1$. By Lemma~\ref{lemma213}, we have $\te{card}(\ga_4\ii J_1)=2$ and $\te{card}(\ga_4\ii [S_2(0),1])=2$.
Let $V(P, \ga_4\ii [S_2(0),1])$ be the quantization error contributed by the set $\ga_4\ii [S_2(0),1]$.
Let $\gb:=\set{a(11), a(11, \infty), a(2), a(2, \infty)}$. The distortion error due to the set $\gb\ii [S_2(0),1]:=\set{a(2), a(2, \infty)}$ is given by
\[\int_{[S_2(0),1]}\min_{a \in \gb\ii [S_2(0),1]}(x-a)^2dP=2 \int_{J_2}(x-a(2))^2 dP=\frac{1}{1296},\]
and so $V(P, \ga_4\ii [S_2(0),1])\leq \frac{1}{1296}=0.000771605$. Suppose that $\ga_4 \ii J_2=\es$, i.e., $S_2(1)<a_3$. Then,
\[V(P, \ga_4\ii [S_2(0),1])\geq \int_{J_2}(x-S_2(1))^2 dP=\frac{1}{864}=0.00115741>V(P, \ga_4\ii [S_2(0),1]),\]
 which is a contradiction. So, $a_3\leq S_2(1)$. We now show that $\ga_4\ii [S_3(0), 1]\neq \es$. Suppose that $\ga_4\ii [S_3(0), 1]=\es$. Then, $a_4<S_3(0)$, and so
\[V(P, \ga_4\ii [S_2(0),1])\geq \int_{J_3\uu J_4\uu J_5}(x-S_3(0))^2 dP=\frac{3959}{5038848}=0.000785695>V(P, \ga_4\ii [S_2(0),1]),\]
which leads to a contradiction. Therefore, $S_3(0)\leq a_4$. Since, $a_3\leq S_2(1)$ and $S_3(0)\leq a_4$, we can assume that $\ga_4$ does not contain any point from the open interval $(S_2(1), S_3(0))$. Since $\frac 12(a_3+a_4)\geq \frac 12(\frac 23+\frac 89)=\frac {7}{9}=S_2(1)$, the Voronoi region of any point in $\ga_4\ii [S_3(0),1]$ does not contain any point from $J_2$. If the Voronoi region of any point in $\ga_4\ii J_2$ contains points from $[S_3(0),1]$, we must have $\frac 12(a_3+a_{4})>\frac 89$ implying $a_{4}>\frac {16}9-a_3\geq \frac {16}9-\frac 79=1$,
which leads to a contradiction. Hence, the Voronoi region of any point in $\ga_4\ii J_2$ does not contain any point from $[S_3(0),1]$.
Thus, the proposition is true for $n=4$. Similarly, we can prove that the proposition is true for $n=5, 6, 7$. We now prove that the proposition is true for all $n\geq 8$.
Let $\ga_n:=\set{a_1, a_2, \cdots, a_n}$ be an optimal set of $n$-means for any $n\geq 8$ such that $0<a_1<a_2<\cdots<a_n<1$. Let $V(P, \ga_n\ii [S_2(0),1])$ be the quantization error contributed by the set $\ga_n\ii [S_2(0),1]$. Set $\gb:=\set{a(111), a(111, \infty), a(12), a(12, \infty), a(21), a(21, \infty), a(3), a(3, \infty)}$. The distortion error due to the set $\gb\ii [S_2(0),1]:=\set{a(21), a(21, \infty), a(3), a(3, \infty)}$ is given by
\[\int_{[S_2(0),1]}\min_{a \in \gb\ii [S_2(0),1]}(x-a)^2dP=2\Big(E(a(21))+E(a(3))\Big)=\frac{1}{11664},\]
and so $V(P, \ga_n\ii [S_2(0),1])\leq\frac{1}{11664}=0.0000857339$.
Suppose that $\ga_n$ does not contain any point from $J_2$. Since by Proposition~\ref{prop211}, the Voronoi region of any point from $\ga_n\ii J_1$ does not contain any point from $[S_2(0),1]$, we have
 \[V(P, \ga_n\ii [S_2(0),1])\geq \int_{J_2}(x-\frac{7}{9})^2 dP=\frac{1}{864}=0.00115741>V(P, \ga_n\ii [S_2(0),1]),\]
 which leads to a contradiction. So, we can assume that $\ga_n\ii J_2\neq \es$. Suppose that $\ga_n\ii [S_{3}(0), 1]= \es$.
Then, $a_n<S_3(0)$, and so
\[V(P, \ga_n\ii [S_2(0),1])\geq \int_{J_3\uu J_4}(x-S_3(0))^2 dP=\frac{131}{279936}=0.000467964>V(P, \ga_n\ii [S_2(0),1]),\]
which gives another contradiction. Therefore, $\ga_n\ii [S_{3}(0), 1]\neq \es$. We now show that $\ga_n\ii (S_2(1), S_3(0))=\es$.
Let $j:=\max\set{i : a_i \leq S_2(1) \te{ for all } 1\leq i\leq n}$, and so $a_j\leq \frac{7}{9}=S_2(1)$. Suppose that $\frac 79<a_{j+1}<\frac 89$.
 Then, the following two cases can arise:

Case~1. $\frac 79<a_{j+1}\leq \frac{5}{6}.$

Then, $\frac 12(a_{j+1}+a_{j+2})>\frac 89$ implying $a_{j+2}>\frac {16}{9}-a_{j+1}\geq \frac {16}{9}-\frac {5}{6}=\frac{17}{18}$, and so
\[V(P, \ga_n\ii [S_2(0),1]) \geq \int_{J_3}(x-\frac{17}{18})^2 dP=\frac{1}{5184}=0.000192901>V(P, \ga_n\ii J_{(1, \infty)}),\]
which is contradiction.

Case~2. $\frac{5}{6} \leq a_{j+1}< \frac 89.$

Then, $\frac 12(a_{j}+a_{j+1})<\frac 79$ implying $a_{j}<\frac {14}9-a_{j+1}\leq \frac {14}9-\frac{5}{6}=\frac{13}{18}<S_{22}(0)$, and so
\[V(P, \ga_n\ii J_{(1, \infty)})\geq \int_{J_{22}\uu J_{23}}(x-\frac{13}{18})^2 dP=\frac{521}{5038848}=0.000103397>V(P, \ga_n\ii J_{(1, \infty)}),\]
which gives a contradiction.
Therefore, $\ga_n\ii (S_{2}(1),  S_{3}(0))\neq \es$. Proceeding similarly, as shown for $n=4$, in this case we can also show that the Voronoi region of any point in $\ga_n\ii J_{2}$ does not contain any point from $[S_{3}(0), 1]$ and the Voronoi region of any point in $\ga_n\ii [S_{3}(0), 1]$ does not contain any point from $J_{2}$.
Thus, the proof of the proposition is complete.
\end{proof}

\begin{prop}\label{prop215}
Let $\ga_n$ be an optimal set of $n$-means for $n\geq 2$. Then, there exists a positive integer $k$ such that  $\ga_n\ii J_j \neq \es$ for all $1\leq j\leq k$, and $\te{card}(\ga_n\ii [S_{k+1}(0), 1])=1$. Moreover, if $n_j:=\te{card}(\ga_j)$, where $\ga_j:=\ga_n\ii J_j$, then $n=\sum_{j=1}^k n_j+1$, with
\[V_n=
\mathop\sum\limits_{j=1}^k p_js_j^2 V_{n_j}+p_k s_k^2 V.\]
\end{prop}

\begin{proof} By Proposition~\ref{prop211} and Proposition~\ref{prop212}, we see that if $\ga_n$ is an optimal set of $n$-means for $n\geq 2$, then $\ga_n\ii J_1\neq \es$, $\ga_n\ii [S_2(0), 1]\neq \es$, and $\ga_n$ does not contain any point from the open interval $(S_1(1), S_2(0))$. Proposition~\ref{prop214} says that if $\te{card}(\ga_n\ii [S_{k+1}(0), 1])\geq 2$ for some $k\in \D N$, then $\ga_n\ii J_{k+1}\neq \es$ and $\ga_n\ii [S_{k+1}(0), 1])\neq \es$. Moreover, $\ga_n$ does not take any point from the open interval $(S_{k+1}(1), S_{k+2}(0))$. Thus, by Induction Principle, we can say that if $\ga_n$ is an optimal set of $n$-means for $n\geq 2$, then there exists a positive integer $k$ such that $\ga_n\ii J_j\neq \es$ for all $1\leq j\leq k$ and $\te{card}(\ga_n\ii [S_{k+1}(0), 1])=1$.

For a given $n\geq 2$, write $\ga_j:=\ga_n\ii J_j$ and $n_j:=\te{card}(\ga_j)$. Since $\ga_j$ are disjoints for $1\leq j\leq k$, and $\ga_n$ does not contain any point from the open intervals $(S_{\ell}(1), S_{\ell+1}(0))$ for $1\leq \ell\leq k$, we have $\ga_n=\mathop{\uu}\limits_{j=1}^k\ga_j\uu \set{a(k, \infty)}$ and $n=n_1+n_2+\cdots +n_k+1$. Then, using Lemma~\ref{lemma1}, we deduce
\begin{align*} &V_n=\int \min_{a \in \ga_n} (x-a)^2 dP=\sum_{j=1}^{k} \int_{J_j}\min_{a \in \ga_j} (x-a)^2 dP+\int_{J_{(k, \infty)}}(x-a(k, \infty))^2 dP\\
&=\sum_{j=1}^{k} p_j\int\min_{a \in \ga_j} (x-a)^2 dP\circ S_j^{-1}(x)+\int_{J_{(k, \infty)}}(x-a(k, \infty))^2 dP,
\end{align*}
which yields
\begin{equation} \label{eq33} V_n=\sum_{j=1}^{k}p_js_j^2\int\min_{a \in S_j^{-1}(\ga_j)} (x-a)^2 dP+ p_k s_k^2 V.\end{equation}
We now show that $S_j^{-1}(\ga_j)$ is an optimal set of $n_j$-means, where $1\leq j\leq k$. If $S_j^{-1}(\ga_j)$ is not an optimal set of $n_j$-means, then we can find a set $\gb \sci \D R$ with $\te{card}(\gb)=n_j$ such that $\int\mathop{\min}\limits_{b\in \gb} (x-b)^2 dP<\int \mathop{\min}\limits_{a\in S_j^{-1}(\ga_j)}(x-a)^2 dP$. But, then $S_j(\gb) \uu (\ga_n\setminus \ga_j)$ is a set of cardinality $n$ such that
\[\int\mathop{\min}\limits_{a\in S_j(\gb) \uu (\ga_n\setminus \ga_j)}(x-a)^2 dP<\int\mathop{\min}\limits_{a\in \ga_n} (x-a)^2 dP,\] which contradicts the optimality of $\ga_n$. Thus, $S_j^{-1}(\ga_j)$ is an optimal set of $n_j$-means for $1\leq j\leq k$. Hence, by \eqref{eq33}, we have
\[V_n=\sum_{j=1}^{k}p_js_j^2V_{n_j}+ p_k s_k^2 V.\]
Thus, the proof of the proposition is complete.
\end{proof}

\begin{prop} \label{prop216}
Let $\ga_n$ be an optimal set of $n$-means for $n\geq 2$. Then, for $c\in \ga_n$, we have $c=a(\go)$, or $c=a(\go, \infty)$ for some $\go \in \D N^\ast$.
\end{prop}
\begin{proof} Let $\ga_n$ be an optimal set of $n$-means for $n\geq 2$ such that $c\in \ga_n$.
By Proposition~\ref{prop215}, there exists a positive integer $k_1$ such that  $\ga_n\ii J_{j_1}\neq \es$ for $1\leq j_1\leq k_1$,  and $\te{card}(\ga_n \ii [S_{k_1+1}(0), 1])=1$, and $\ga_n$ does not contain any point from the open intervals $(S_{\ell}(1), S_{\ell+1}(0))$ for $1\leq \ell\leq k_1$. If $c\in \ga_n \ii [S_{k_1+1}(0), 1]$, then $c=a(k_1, \infty)$. If $c\in \ga_n\ii J_{j_1}$ for some $1\leq j_1\leq k_1$ with $\te{card}(\ga_n\ii J_{j_1})=1$, then $c=a(j_1)$. Suppose that $c \in \ga_n\ii J_{j_1}$ for some $1\leq j_1\leq k_1$ and $\te{card}(\ga_n\ii J_{j_1})\geq 2$. Then, as similarity mappings preserve the ratio of the distances of a point from any other two points, using Proposition~\ref{prop215} again, there exists a positive integer $k_2$ such that  $\ga_n\ii J_{j_1j_2}\neq \es$ for $1\leq j_2\leq k_2$, and $\te{card}(\ga_n \ii [S_{j_1(k_2+1)}(0), 1])=1$, and $\ga_n$ does not contain any point from the open intervals $(S_{j_1\ell}(1), S_{j_1(\ell+1)}(0))$ for $1\leq \ell\leq k_2$. If $c\in \ga_n \ii [S_{j_1(k_2+1)}(0), 1]$ then $c=a(j_1k_2, \infty)$. Suppose that $c\in \ga_n\ii J_{j_1j_2}$ for some $1\leq j_2\leq k_2$. If $\te{card}(\ga_n\ii J_{j_1j_2})=1$, then $c=a(j_1j_2)$. If $\te{card}(\ga_n\ii J_{j_1j_2})\geq 2$, proceeding inductively as before, we can find a word $\go \in \D N^\ast$, such that either $c \in \ga_n\ii J_\go$ with $\te{card}(\ga_n\ii J_\go)=1$ implying $c=a(\go)$, or  $c\in \ga_n\ii [S_{\go^-(\go_{|\go|}+1)}, 1]$ with $\te{card}(\ga_n\ii [S_{\go^-(\go_{|\go|}+1)}, 1])=1$ implying $c=a(\go, \infty)$. Thus, the proof of the proposition is complete.
\end{proof}

\begin{prop} \label{prop216}
For any $n\geq 2$, let $\ga_n$ be an optimal set of $n$-means with respect to the probability distribution $P$. Write
\begin{align*}
&W(\ga_n):=\set{\go\in \D N^\ast : a(\go) \te{ or } a(\go, \infty) \in \ga_n}, \te{ and }\\
& \tilde W(\ga_n):=\set{\gt \in W(\ga_n) : p_\gt s_\gt^2 \geq p_\go s_\go^2 \te{ for all } \go \in W(\ga_n)}.
\end{align*}
Then, for any  $\gt\in \tilde W(\ga_n)$ the set $\ga_{n+1}:=\ga_{n+1}(\gt)$, where
\[\ga_{n+1}(\gt)=\left\{\begin{array}{ll}
(\ga_n\setminus \set{a(\gt)})\uu \set{a(\gt1), a(\gt1, \infty)} \te{ if } a(\gt) \in \ga_n,  & \\
(\ga_n\setminus \set{a(\gt, \infty)})\uu \set{a(\gt^-(\gt_{|\gt|}+1)), a(\gt^-(\gt_{|\gt|}+1), \infty)} \te{ if } a(\gt, \infty) \in \ga_n,  &
\end{array}\right.\]
is an optimal set of $(n+1)$-means.
\end{prop}

\begin{proof}
Let us first claim that for any $\go, \gt \in \D N^\ast$, $p_\gt s_\gt^2 \geq p_\go s_\go^2$ if and only if
\begin{align*} E(a(\gt1)) +E(a(\gt1, \infty))+E(a(\go))\leq E(a(\gt))+E(a(\go1))+E(a(\go1, \infty)).\end{align*}
By Lemma~\ref{lemma4}, we have
\begin{align*} &LHS=2p_{\gt1}s_{\gt1}^2 V +p_\go s_\go^2 V=\frac{1}{9}p_{\gt}s_{\gt}^2 V +p_\go s_\go^2 V,\\
&RHS=p_{\gt}s_{\gt}^2 V +2 p_{\go1} s_{\go1}^2 V=p_{\gt}s_{\gt}^2 V +\frac{1}{9}p_\go s_\go^2 V.
\end{align*}
Thus, $LHS\leq RHS$ if and only if $p_\gt s_\gt^2 \geq p_\go s_\go^2$, which is the claim.

We now prove the proposition by induction. By Lemma~\ref{lemma11}, we know that the optimal set of two-means is $\ga_2=\set{a(1), a(1, \infty)}$. Here $\tilde W(\ga_2)=W(\ga_2)=\set{1}$. Since $a(1) \in \ga_2$, we have $\ga_3=\set{a(11), a(11, \infty), a(1, \infty)}$. Again, as $a(1, \infty) \in \ga_2$, we have $\ga_3=\set{a(1), a(2), a(2, \infty)}$. Clearly by Lemma~\ref{lemma12}, the sets $\ga_3$ are optimal sets of three-means. Thus, the proposition is true for $n=2$.
Let us now assume that $\ga_m$ is an optimal set of $m$-means for some $m\geq 2$. Write
\begin{align*}
&W(\ga_m):=\set{\go\in \D N^\ast : a(\go) \te{ or } a(\go, \infty) \in \ga_m}, \te{ and }\\
& \tilde W(\ga_m):=\set{\gt \in W(\ga_m) : p_\gt s_\gt^2 \geq p_\go s_\go^2 \te{ for all } \go \in W(\ga_m)}.
\end{align*}
If $\gt \not \in \tilde W(\ga_m)$, i.e., if  $\gt \in  W(\ga_m)\setminus \tilde W(\ga_m)$, then by the claim, if $a(\gt) \in \ga_m$ the error
\[\int\min \set{(x-a)^2 : a\in (\ga_m\setminus \{a(\gt)\})\uu \{a(\gt 1), a(\gt1, \infty)\}}dP,\] or, if $ a(\gt, \infty) \in \ga_m$ the error
\[\int\min \set{(x-a)^2 : a\in (\ga_m\setminus \{a(\gt, \infty)\})\uu \{a(\gt^-(\gt_{|\gt|}+ 1)), a(\gt^-(\gt_{|\gt|}+ 1), \infty)\}}dP\]
is either equal or larger, in fact strictly larger if $n$ is not of the form $2^k$ for any positive integer $k$, than the corresponding error obtained in the case where $\gt\in \tilde W(\ga_m)$. Hence, for any  $\gt\in \tilde W(\ga_n)$ the set $\ga_{m+1}:=\ga_{m+1}(\gt)$, where
\[\ga_{m+1}(\gt)=\left\{\begin{array}{ll}
(\ga_m\setminus \set{a(\gt)})\uu \set{a(\gt1), a(\gt1, \infty)} \te{ if } a(\gt) \in \ga_m,  & \\
(\ga_m\setminus \set{a(\gt, \infty)})\uu \set{a(\gt^-(\gt_{|\gt|}+1)), a(\gt^-(\gt_{|\gt|}+1), \infty)} \te{ if } a(\gt, \infty) \in \ga_m,  &
\end{array}\right.\]
is an optimal set of $(m+1)$-means. Thus, by the principle of mathematical induction, the proposition is true for all positive integers $n\geq 2$.
\end{proof}

\begin{lemma} \label{lemma14} Let $n\in \D N$ be such that $n=2^k$ for some $k\geq 1$. Then,
\[\ga(k):=\set{a(\go) : p_\go=\frac{1}{2^k}}\uu \set{a(\go, \infty) : p_\go=\frac1{2^k}}\] is an optimal set of $n$-means.
Set $\ga_j(k):=\ga(k)\ii J_j$ for $1\leq j\leq k$. Then, $S_j^{-1}(\ga_j(k))$ is an optimal set of $2^{k-j}$-means for $1\leq j\leq k$.  Moreover, $n=\sum_{j=1}^k 2^{k-j}+1$ and \[V_n=\sum_{j=1}^k  \frac 1{18^j} V_{2^{k-j}}+\frac 1{18^k} V_1.\]
\end{lemma}

\begin{proof} Let us prove the lemma by induction. If $n=2$, i.e., when $k=1$,  by Lemma~\ref{lemma11}, we have $\ga(1)=\set{a(1), a(1, \infty)}=\set{a(\go) : p_\go=\frac 12}\uu \set{a(\go, \infty) : p_\go=\frac 1 2}$ which is an optimal set of two-means. Here $\ga_1(1)=\ga(1)\ii J_1=\set{a(1)}$. Notice that $\te{card}(\ga_1(1))=1$, and the set $S_1^{-1}(\ga_1)=\set{\frac 12}$ is an optimal set of one-mean. Moreover,
$V_2=\frac 1{18}V_1+\frac 1{18} V_1$. Thus, the lemma is true for $n=2$.
Let the lemma be true if $n=2^k$ for some $k=m$, where $m\in \D N$ and $m\geq 2$. We will show that it is also true for $k=m+1$. We have
\[\ga(m)= \set{a(\go) : p_\go=\frac{1}{2^m}}\uu \set{a(\go, \infty) : p_\go=\frac1{2^m}}.\]
List the elements of $\ga(m)$ as $a_1, a_2, \cdots, a_{2^m}$, i.e., $\ga(m)=\set{a_j : 1\leq j\leq 2^m}$. Construct the sets $A_j$ for $1\leq j\leq 2^m$ as follows:
\[
A_j:=\left\{\begin{array} {ll}
\set{a(\go1), a(\go1, \infty)} \te{ if } a_j=a(\go) \te{ for some } \go \in \D N^\ast, & \\
\set {a(\go^-(\go_{|\go|}+1)), a(\go^-(\go_{|\go|}+1), \infty)} \te{ if } a_j=a(\go, \infty) \te{ for some } \go \in \D N^\ast. &
\end{array}\right. \]
For $1\leq j\leq 2^m$, set $\ga_{2^m+j}=(\ga(m)\setminus \mathop{\uu}\limits_{k=1}^j\set{a_k})\uu A_1\uu A_2\uu \cdots \uu A_j.$
Since $\ga_{2^m}$ is an optimal set of $2^{m}$-means, by Proposition~\ref{prop216}, $\ga_{2^{m}+1}$ is an optimal set of $(2^m+1)$-means, which implies  $\ga_{2^m+2}$ is an optimal set of $(2^m+2)$-means, and thus proceeding inductively, we can say that
the set
\[\ga_{2^{m+1}}:=\ga_{2^m+2^m}=(\ga(m)\setminus \uu_{k=1}^{2^m}\set{a_k})\uu A_1\uu A_2\uu \cdots \uu A_{2^m}=A_1\uu A_2\uu \cdots \uu A_{2^m}\] is an optimal set of $2^{m+1}$-means. Notice that for any $\go \in \D N^\ast$ if $a(\go)$ or $a(\go, \infty) \in A_j$, then $p_\go=\frac 1{2^{m+1}}$, and so
\[\ga_{2^{m+1}}=\ga(m+1)=\set{a(\go) : p_\go=\frac{1}{2^{m+1}}}\uu \set{a(\go, \infty) : p_\go=\frac1{2^{m+1}}}.\]
Therefore, by using the principle of mathematical induction, we can say that the set $\ga(k)$ is an optimal set of $n$-means if $n\in \D N$ and $n=2^k$ for some $k\geq 1$.
To complete the rest of the proof, we proceed as follows:
For any $\go=\go_1\go_2\cdots \go_{|\go|} \in \D N^\ast$, we have  $a(\go):=S_\go(\frac 12) \in J_{\go_1}$. Again, from the definitions of $a(\go)$, $a(\go, \infty)$, if $a(\go) \in J_{\go_1}$ and $|\go|>1$, then $a(\go, \infty)\in J_{\go_1}$. Keeping $\go_1$ fixed, if $\go_1< k$, we see that there are $2^{k-\go_1-1}$ different $\gt \in \D N^\ast$ such that $p_{\go_1\gt}=\frac 1{2^{k}}$. Thus,  for any $\go=\go_1\go_2\cdots\go_{|\go|} \in \D N^\ast$ with $|\go|>1$ and $p_\go=\frac 1{2^{k}}$, the optimal set $\ga(k)$ contains $2^{k-\go_1}$ elements from $J_{\go_1}$; in other words, $\te{card}(\ga(k)\ii J_{\go_1})=2^{k-\go_1}$. If $|\go|=1$ and $p_\go=\frac 1{2^{k}}$, i.e., when $\go=k$, then $a(k) \in J_{k}$, i.e., $\ga(k)$ contains only one element from $J_{k}$. Besides, $\ga(k)$ contains the element $a(k, \infty)$.
Write $\ga_j(k):=\ga(k)\ii J_j$. Then, $\te{card}(\ga_j(k))=2^{k-j}$ for $1\leq j\leq k$. For any $1\leq j\leq k-1$, by the definition of the mappings,  we have
\begin{align*}
S_j^{-1}(\ga_j(k))&=\set{a(\go_{j+1} \cdots \go_{|\go|}) : p_{\go_{j+1} \cdots \go_{|\go|}}=\frac{1}{2^{k-j}}}\uu \set{a(\go_{j+1} \cdots \go_{|\go|}, \infty) : p_{\go_{j+1} \cdots \go_{|\go|}}=\frac1{2^{k-j}}},
\end{align*}
 and $ S_k^{-1}(\ga_k(k))=\set{\frac 12}$. Thus, for all $1\leq j\leq k$, we can see that $S_j^{-1}(\ga_j(k))=\ga(k-j)$. Hence, by the first part of the lemma,  for each $1\leq j\leq k$, the set $S_j^{-1}(\ga_j(k))$ is an optimal set of $2^{k-j}$-means.
Now,
\begin{align*} &V_n=\int \min_{a \in \ga(k)}  \|x-a\|^2 dP=\sum_{j=1}^{k} \int_{J_j}\min_{a \in \ga_j(k)} (x-a)^2 dP+\int_{J_{(k, \infty)}}(x-a(k, \infty))^2 dP\\
&=\sum_{j=1}^{k} p_j\int\min_{a \in \ga_j(k)} (x-a)^2 dP\circ S_j^{-1}(x)+\int_{J_k}(x-a(k))^2 dP,
\end{align*}
which yields
\[V_n=\sum_{j=1}^{k}\frac{1}{18^j} \int\min_{a \in S_J^{-1}(\ga_j(k))} (x-a)^2 dP+\frac 1{18^k} V_1=\sum_{j=1}^{k}\frac{1}{18^j}V_{2^{k-j}}+\frac 1{18^k}V_1.\]
Thus, the proof of the lemma is complete.
\end{proof}

\begin{remark}
The set $\ga(k))$ given by Lemma~\ref{lemma14} is a unique optimal set of $n$-means where $n=2^k$ for some $k\in \D N$.
\end{remark}

In regard to Lemma~\ref{lemma14} let us give the following example.
\begin{example} Take $n=16=2^4$. Then,
\begin{align*} &\ga(4)=\set{a(1111), a(1111, \infty), a(112), a(112, \infty), a(121), a(121, \infty), a(13), \\
& \qquad a(13, \infty), a(211), a(211, \infty), a(22), a(22, \infty), a(31), a(31, \infty), a(4), a(5, \infty)}.
\end{align*} Since, $\ga_j(4)=\ga(4)\ii J_j$ for $1\leq j\leq 4$, we have
\begin{align*} &\ga_1(4)=\set{a(1111), a(1111, \infty), a(112), a(112, \infty), a(121), a(121, \infty), a(13), a(13, \infty)},\\
&\ga_2(4)=\set{ a(211), a(211, \infty), a(22), a(22, \infty)},\\
&\ga_3(4)=\set{(31), a(31, \infty)},\\
&\ga_4(4)=\set{a(4)}.
\end{align*}
Here, $S_1^{-1}(\ga_1(4))=\set{a(111), a(111, \infty), a(12), a(12, \infty), a(21), a(21, \infty), a(3), a(3, \infty)}$ is an optimal set of $2^3$-means, $S_2^{-1}(\ga_2(4))=\set{a(11), a(11, \infty), a(2), a(2, \infty)}$ is an optimal set of $2^2$-means, $S_3^{-1}(\ga_3(4))=\set{a(1), a(1, \infty)}$ is an optimal set of $2$-means, and $S_4^{-1}(\ga_4(4))=\set{\frac 12}$ is an optimal set of one-mean. Moreover, we can see that
\[V_{16}=\frac 1{18} V_8+\frac 1{18^2} V_4+\frac 1{18^3} V_2+\frac 1{18^4} V_1+\frac {1}{18^4} V_1.\]
\end{example}

Let us now state and prove the main theorem of the paper.
\begin{theorem} \label{Th1}
For $n\in \D N$ with $n\geq 2$ let $\ell(n)\in \D N $ satisfy $2^{\ell(n)} \leq n<2^{\ell(n)+1}$. Let $\ga(\ell(n))$ and $\ga_n(I)$ be the sets as defined by Definition~\ref{defi1}. Then, $\ga_n(I)$ is an optimal set of $n$-means with quantization error
\[V_n=\frac 1{18^{\ell{(n)}}}  \frac 1 8\Big(2^{\ell(n)+1}-n+\frac 1 9(n-2^{\ell(n)})\Big).\]
The number of such sets is ${}^{2^{\ell(n)}}C_{n-2^{\ell(n)}}$, where ${}^uC_v={u\choose v}$ is a binomial coefficient.
\end{theorem}

\begin{proof}
By Lemma~\ref{lemma14}, $\ga(\ell(n))$ is an optimal set of $2^{\ell(n)}$-means. Choose $I\sci \ga(\ell(n))$ such that $\te{card}(I)=n-2^{\ell(n)}$. List the elements of $I$ as $a_1, a_2, \cdots, a_{n-2^{\ell(n)}}$, i.e., $I=\set{a_j : 1\leq j\leq n-2^{\ell(n)}}$. Construct the sets $A_j$ for $1\leq j\leq n-2^{\ell(n)}$ as follows:
\[
A_j:=\left\{\begin{array} {ll}
\set{a(\go1), a(\go1, \infty)} \te{ if } a_j=a(\go) \te{ for some } \go \in \D N^\ast, & \\
\set {a(\go^-(\go_{|\go|}+1)), a(\go^-(\go_{|\go|}+1), \infty)} \te{ if } a_j=a(\go, \infty) \te{ for some } \go \in \D N^\ast. &
\end{array}\right. \]
For $1\leq j\leq n-2^{\ell(n)}$, set \[\ga_{2^{\ell(n)}+j}=(\ga(\ell(n))\setminus \mathop{\uu}\limits_{k=1}^j\set{a_k})\uu A_1\uu A_2\uu \cdots \uu A_j.\]
As shown in Lemma~\ref{lemma14}, proceeding inductively, we see that
the set  $\ga_n(I):=\ga_{2^{\ell(n)}+(n-2^{\ell(n)})}=(\ga(\ell(n))\setminus I)\uu A_1\uu A_2\uu \cdots \uu A_{n-2^{\ell(n)}}$ forms an optimal set of $n$-means. Then, using Proposition~\ref{prop1}, we obtain the quantization error as
\[V_n=\int\min_{a\in \ga_n}(x-a)^2 dP=\frac 1{18^{\ell{(n)}}}  \frac 1 8\Big(2^{\ell(n)+1}-n+\frac 1 9(n-2^{\ell(n)})\Big).\]
Since the subset $I$ from the set $\ga(\ell(n))$ can be chosen in ${}^{2^{\ell(n)}}C_{n-2^{\ell(n)}}$ different ways, the number of $\ga_n(I)$ is ${}^{2^{\ell(n)}}C_{n-2^{\ell(n)}}$.
Thus, the proof of the theorem is complete.
\end{proof}

\end{document}